\documentclass[preprint,10pt]{article}
\usepackage{comment}

\usepackage[english]{babel}
\usepackage[T1]{fontenc}
\usepackage[utf8]{inputenc}
\usepackage{amsmath,amssymb,mathrsfs,enumerate}
\usepackage{amscd}
\usepackage{a4wide}
\usepackage{tikz}
\usepackage{commath}
\date{}

\usepackage{hyperref}

\numberwithin{equation}{section}
\newtheorem{theo}{Theorem}[section]
\newtheorem{cor}{Corollary}[section]
\newtheorem{prop}{Proposition}[section]
\newtheorem{lem}{Lemma}[section]

\numberwithin{figure}{section}

\newcommand{\dN}{{\partial_\nu}}
\newcommand{\pR}{{\Phi_\rho}}
\newcommand{\bn}{{\mathbb{N}}}

\newcommand{\br}{{\mathbb{R}}}
\newcommand{\bc}{{\mathbb{C}}}

\newenvironment{proof}
	{\textbf{Proof.}}
	{\hfill $\square$\vskip 8pt}

\title{Logarithmic stability inequality in an inverse source problem for the heat equation on a waveguide}
\author{
Yavar {\sc Kian}$\footnote{Aix Marseille Universit\'e, Univ. Toulon, CNRS, CPT, Marseille, France, 
yavar.kian@univ-amu.fr}$, 
Diomba {\sc Sambou}$\footnote{Departamento de Matem\'aticas, Facultad de 
Matem\'aticas, Pontificia Universidad Cat\'olica de Chile, Vicu\~na Mackenna 
4860, Santiago de Chile, disambou@mat.uc.cl}$,
and Eric {\sc Soccorsi}$\footnote{Aix Marseille Universit\'e, Univ. Toulon, CNRS, CPT, Marseille, France, eric.soccorsi@univ-amu.fr}$
}

\begin{document}

\maketitle

\begin{abstract} We prove logarithmic stability in the parabolic inverse problem of determining the space-varying factor in the source, by a single partial boundary measurement of the solution to the heat equation in an infinite closed waveguide, with homogeneous initial and Dirichlet data.
\end{abstract}

\bigskip


\bigskip



\section{Introduction}


\subsection{Statement and origin of the problem}

Let $\omega \subset \mathbb{R}^{n - 1}$, $n \geq 2$, be open and connected, with $\mathcal{C}^4$ boundary $\partial \omega$. Set $\Omega := \omega \times \br$ and $\Gamma := \partial \omega \times \br$. For $T \in (0,+\infty)$ fixed, we consider the parabolic initial boundary value problem (IBVP)
\begin{equation}
\label{eq:bvp}
\begin{cases} 
\partial_t u - \Delta u = F(t,x) & \text{in}\ Q := (0,T) \times \Omega, \\ 
u(0,\cdot) = 0 & \text{in}\ \Omega, \\
u = 0 & \text{on}\ \Sigma := (0,T) \times \Gamma,
\end{cases}
\end{equation}
with source term $F \in L^2(Q)$. 
In this paper, we examine the inverse problem of determining $F$ from a single Neumann boundary measurement of the solution $u$ to \eqref{eq:bvp}. 

Let us first notice that there is a natural obstruction to uniqueness in this problem. This can be easily understood from the identity
$\partial_\nu u=0$ on $\Sigma$, verified by any $u \in \mathcal C^\infty_0(Q)$, despite of the fact that the function $F:=(\partial_t-\Delta)u$ may well be non uniformly zero in $Q$. Otherwise stated, the observation of $\partial_\nu u$ on $\Sigma$ may be unchanged, whereas $F$ is modified. To overcome this problem, different lines of research can be pursued. One of them is to extend the set of data available in such a way that $F$ is uniquely determined by these observations. Another direction is the one of assuming that the source term $F$ is {\it a priori} known to have the structure
\begin{equation}
\label{eq-s}
F(t,x)=\sigma(t)\beta(x), \quad (t,x) \in Q, 
\end{equation}
where $t \mapsto \sigma(t)$ is a known function, and then proving that $\dN u$ uniquely determines $\beta$.
In this paper we investigate the second direction. Namely,  we examine the stability issue in the identification of the
time-independent part  $\beta$ of the source, from partial observation on $\Sigma$ of the flux $\dN u$ induced by the solution $u$ to \eqref{eq:bvp}. 

Source terms of the form \eqref{eq-s} are commonly associated with the reaction term in linear reaction diffusion equations. These equations arise naturally in various fields
of application, investigating systems made of several interacting components, such as population dynamics \cite{M}, 
fluid dynamics \cite{B}, or heat conduction \cite{BBC}. More precisely, when $\sigma(t):=e^{-\mu t}$, where $\mu$ a positive constant, the system \eqref{eq:bvp}-\eqref{eq-s} describes the diffusion in transmission lines or cooling pipes with significantly large length-to-diameter ratio, of decay heat, that is the heat released as a result of radioactive decay.
In this particular case, \eqref{eq-s} models a heat source produced by the decay of a radioactive isotope, and $\beta$ is the spatial density of the isotope. From a practical viewpoint, the rate of decay $\mu$ of the isotope inducing the decay heat diffusion process, is known, and therefore the same is true for the function $\sigma$, while the density function $\beta$ is generally unknown.
This motivates for a closer look into the inverse problem under investigation in this article.



\subsection{Existing papers: a short review}

Inverse source problems have been extensively studied over the last decades. We refer to \cite{I} for a more general overview of this topic than the one presented in this section, where we solely focus on parabolic inverse source problems consisting in determining a source term by boundary measurements of the solution to a parabolic equation. The uniqueness issue for this problem was investigated in \cite{C}, and conditional stability was derived in
\cite{choY,yam,yam1}. In \cite{IY}, inspired by the Bukhgeim-Klibanov approach introduced in \cite{BK}, Imanuvilov and Yamamoto proved Lipschitz stability of the source with respect to one Neumann boundary measurement of the solution to a parabolic equation with non-degenerate initial data, and partial Dirichlet data supported on arbitrary subregions of the boundary. In \cite{cho}, Choulli and Yamamoto established a log-type stability estimate for the time-independent source term $\beta$, appearing in \eqref{eq-s}, by a single Neumann observation of the solution on an arbitrary sub-boundary. 

All the above mentioned results are stated in a bounded spatial domain. But, to the best of our knowledge, there is no result available in the mathematical literature, dealing with the recovery of a non-compactly supported unknown source function, appearing in a parabolic equation, by boundary measurements of the solution. 
This is the starting point of this paper, in the sense that we aim for extending the stability result of \cite{cho}, which is valid in bounded spatial domains only, to the framework of infinite cylindrical domains. 

\subsection{The forward problem}

Prior to describing the main achievement of this paper in section \ref{sec-mr}, we briefly investigate the well-posedness of the IBVP\eqref{eq:bvp}. Actually, we start by examining the forward problem associated with the IBVP
\begin{equation}
\label{eq:bvp1}
\begin{cases} 
\partial_t v - \Delta v = f & \text{in}\ Q, \\ 
v(0,\cdot) = v_0 & \text{in}\ \Omega,\\
v = 0 & \text{on}\ \Sigma,
\end{cases}
\end{equation}
for suitable source term $f$ and initial data $v_0$. More precisely, we seek an existence, uniqueness and (improved) regularity result for the solution to the above system, as well as a suitable energy estimate. Such results are rather classical in the case of bounded spatial domains, but it turns out that they are not so well-documented for unbounded domains such as $\Omega$. Therefore, for the sake of completeness, we shall establish Theorem \ref{t1}, presented below. 

As a preliminary,  we recall for all $s\in(0,+\infty)$, that the Sobolev space
$$ H^s(\Omega):=L^2(\br;H^s(\omega))\cap H^s(\br;L^2(\omega))$$
is endowed with the norm
$$ \|u\|_{H^s(\Omega)}^2:=\|u\|_{L^2(\br;H^s(\omega))}^2+\|u\|_{H^s(\br;L^2(\omega))}^2, $$
and then remind from \cite[Section 4.2.1]{LM2} that
$$
H^{r,s}(Q) := L^2 \big(0,T;H^r(\Omega)) \cap 
H^s(0,T;L^2(\Omega)),\ r \in ( 0, +\infty).
$$
Having said that, we may now state the following forward result associated with \eqref{eq:bvp1}.

\begin{theo}
\label{t1}
Let $v_0 \in H_0^1(\Omega)$ and $f \in L^2(0,T;H_0^1(\Omega))$. Then, there exists a unique solution $v \in H^{2,1}(Q) \cap \mathcal C([0,T];H^1_0(\Omega))$ to the IBVP \eqref{eq:bvp1}, such that
\begin{equation}
\label{es1}
\Vert v(t) \Vert_{H^1(\Omega)} \leq 
\Vert v_0 \Vert_{H^1(\Omega)} + \Vert f \Vert_{L^2(0,T;H^1(\Omega))},\ t\in[0,T].
\end{equation}
\end{theo}

As will be seen in the sequel, Theorem \ref{t1} is a crucial step in the derivation of the observability inequality presented in Section \ref{sec-obs}, which is a cornerstone in the analysis of the inverse problem under investigation. But, just as important is the following consequence of Theorem \ref{t1}, which enables us to define properly the boundary data used by the identification of the unknown function $\beta$ in Theorem \ref{t2}, below. 

\begin{cor}
\label{cor-e}
Let $F$ be defined by \eqref{eq-s}, where $\sigma \in \mathcal C^1([0,T])$ and $\beta \in H_0^1(\Omega)$. Then, the IBVP \eqref{eq:bvp1} admits a unique solution $u \in H^{2,1}(Q)$. Moreover, we have $\partial_t u \in H^{2,1}(Q) \cap \mathcal C([0,T];H^1_0(\Omega))$, and the following estimate holds:
\begin{equation}
\label{es2}
\Vert \partial_t u(t) \Vert_{H^1(\Omega)} \leq (1+T^{1 \slash 2}) \Vert \sigma \Vert_{\mathcal C^1([0,T])} \Vert \beta \Vert_{H^1(\Omega)},\ t \in [0,T].
\end{equation}
\end{cor}

The  proofs of Theorem \ref{t1} and Corollary \ref{cor-e} are presented in Section \ref{sec-fwd}.


\subsection{Main result}
\label{sec-mr}

For $M \in (0,+\infty)$ fixed, we introduce the set of admissible unknown source functions, as
\begin{equation}
B(M) :=  \left\{ \varphi \in H^{1}_0(\Omega);\ \norm{\varphi}_{H^1(\Omega)} \leq M \right\}.
\end{equation}
Then, the main result of this article can be stated as follows.

\begin{theo}
\label{t2} 
Put $\gamma := \gamma' \times \br$, where $\gamma'$ is an arbitary closed subset of the boundary $\partial \omega$,  with non empty interior, and let $\sigma \in \mathcal C^1([0,T])$ satisfy $\sigma(0) \neq 0$. For $M \in (0,+\infty)$, pick $\beta \in B(M)$, and let $u$ be the $H^{2,1}(Q)$-solution to the IBVP \eqref{eq:bvp}, associated with 
$$
F(t,x)=\sigma(t) \beta(x),\ (t,x) \in Q, 
$$
which is given by Corollary \ref{cor-e}.
Then, there exists a constant $C>0$, depending only on $\omega$, $\sigma$, $T$, $M$ and $\gamma'$, such that the estimate
\begin{equation}
\label{eq:in}
\Vert \beta \Vert_{L^2(\Omega)} \leq C \Phi \left( \Vert \partial_\nu u \Vert_{H^1 \left( 0,T;L^2(\gamma) \right)} \right),
\end{equation}
holds with
\begin{equation}
\label{eq:in2}
\Phi(r) := \left\{ \begin{array}{ll} r^{{1 \slash 2}}+ | \ln r |^{-{1 \slash 2}} & \text{if}\ r \in (0,+\infty) \\ 0 & \text{if}\ r =0. \end{array} \right.
\end{equation} 
\end{theo}

Notice that we have $u \in H^1(0,T;H^2(\Omega))$ from Corollary \ref{cor-e}, which guarantees that the trace $\partial_\nu u$ appearing in the right hand side of the stability estimate \eqref{eq:in2} is well-defined in $H^1( 0,T;L^2(\gamma))$.

To the best of our knowledge, Theorem \ref{t2} is the first stability result in the identification of the non-compactly supported source term $\beta$, appearing in a parabolic equation, by a single partial boundary observation of the solution. A similar statement was actually derived in \cite[Theorem 2.2]{cho} (see also \cite[Theorem 3.4]{cho1}) when the domain $\Omega$ is bounded, so Theorem \ref{t2} extends this result to the case of infinite cylindrical domains. 

Notice that the statement of Theorem \ref{t2} is valid in absence of any assumption on the behavior of the source term $\beta$ outside a compact subset of the infinite cylindrical domain $\Omega=\omega\times\br$. Another remarkable feature of the result of Theorem \ref{t2} is that the logarithmic dependency of the space-varying source term, with respect to the boundary data, manifested in \cite[Theorem 2.2]{cho} for a bounded domain, is preserved by the stability estimate \eqref{eq:in}. 
Otherwise stated, the stability of the reconstruction of $\beta$ by a single boundary observation of the solution, is not affected by the infinite extension of the support of the unknown coefficient. This phenomenon is in sharp contrast with the one observed for the determination of the electric potential appearing in the Schr\"odinger equation, by a finite number of Neuman data, where Lipschitz stability (see \cite[Theorem 1]{BP} and \cite[Theorem 1]{BP2}) degenerates to H\"older (see \cite[Theorem 1.4]{KPS2}), as the support of the unknown potential becomes infinite.

The proof of Theorem \ref{t2} is by means of a Carleman inequality specifically designed for the heat operator in the unbounded cylindrical domain $\Omega$. The derivation of this estimate is inspired by the approach used in this particular framework by \cite{BKS,KPS1,KPS2} for the Schr\"odinger equation.


\subsection{Outline}

The paper is organized as follows. In Section \ref{sec-fwd}, we establish Theorem \ref{t1} and Corollary \ref{cor-e}. In Section \ref{sec-CI}, we derive a Carleman estimate for the heat operator in $\Omega$, which is the main tool
for the proof of the observability inequality presented in Section \ref{sec-obs}. Finally, Section \ref{sec-prmain} contains the proof of Theorem \ref{t2}, which is by means of the above mentioned observability inequality.


\section{Analysis of the forward problem}
\label{sec-fwd}


\subsection{Proof of Theorem \ref{t1}}
\label{sec-pr-t1}
Let $A$ be the Dirichlet Laplacian in $L^2(\Omega)$, i.e. the self-adjoint operator generated in $L^2(\Omega)$ by the closed quadratic form
$$ a(\varphi) := \int_{\Omega} \vert \nabla \varphi \vert^2\ dx,\ \varphi \in D(a):=H_0^1(\Omega). $$
Evidently, $A$ acts as $-\Delta$ on its domain $D(A):=\{ \varphi \in H_0^1(\Omega),\ \Delta \varphi \in L^2(\Omega) \}$, so
we may rewrite the IBVP \eqref{eq:bvp1} into the following Cauchy problem
\begin{equation}
\label{eq2}
\begin{cases} 
v' + A v=  & f\ \text{in}\ (0,T), \\ 
v(0) = v_0. &
\end{cases}
\end{equation}
Further, since $\omega$ is bounded, then there exists a positive constant $c(\omega)$, depending only on $\omega$, such that the Poincar\'e inequality holds:
$$ a(\varphi) \geq c(\omega) \norm{\varphi}_{L^2(\Omega)}^2,\ \varphi \in H_0^1(\Omega). $$
As a consequence, we have $A \geq c(\omega)$ in the operator sense, and $A+p$ is thus boundedly invertible in $L^2(\Omega)$ for all $p \in \bc \setminus (0,+\infty)$, with
$$  \norm{(A+p)^{-1}}_{\mathcal B(L^2(\Omega))} \leq \frac{\sqrt{2} \max(1,c(\omega)^{-1})}{1 + \vert p \vert},\ p \in \bc,\ \Re\ p \geq 0. $$
Moreover, as $v_0 \in H_0^1(\Omega)=D(a)=D(A^{1 \slash 2})$ and $f \in L^2(Q)$, then we derive from \cite[Section 4, Theorem 3.2]{LM2} that \eqref{eq2} admits a unique solution 
$$ v \in L^2(0,T;D(A))\cap H^1(0,T;L^2(\Omega)). $$
From this and the identity $D(A)=H^2(\Omega)\cap H^1_0(\Omega)$, arising from \cite[Lemma 2.2]{CKS}, it then follows that $v \in H^{2,1}(Q)$. 

The last step of the proof is to show that $v \in \mathcal C([0,T];H^1_0(\Omega))$ satisfies \eqref{es1}. Bearing in mind that $A$ is the infinitesimal generator of a contraction semi-group of class $\mathcal C^0$ in $L^2(\Omega)$
(see e.g. \cite[Section XVII.A.3, Theorem 2]{DL5}), this can be done with the aid of Duhamel's formula, giving
\begin{equation}
\label{eq-du} 
v(t) = e^{-t A} v_0 + \int_0^t e^{-(t-s) A} f(s)\ ds,\ t \in [0,T].
\end{equation}
Indeed, since $v_0 \in D(A^{1 \slash 2})$ and $e^{-t A}$ commutes with $A^{1 \slash 2}$ for all $t \in [0,T]$, then we have $e^{-t A} v_0 \in D(A^{1 \slash 2})$ and  $A^{1 \slash 2} e^{-t A} v_0=e^{-t A} A^{1 \slash 2} v_0$, which  
entails
\begin{equation}
\label{eq-cty1} 
 t \mapsto e^{-t A} v_0 \in \mathcal C([0,T];D(A^{1 \slash 2}))=\mathcal C([0,T];H^1_0(\Omega)).
\end{equation}
Furthermore, for all $t\in[0,T]$, we infer  from the basic identity
$$ \norm{e^{-t A} v_0}_{H^1(\Omega)}^2  =  \norm{e^{-t A} v_0}_{L^2(\Omega)}^2 + \norm{A^{1 \slash 2} e^{-t A} v_0}_{L^2(\Omega)}^2, $$ 
that
\begin{eqnarray}
\norm{e^{-t A} v_0}_{H^1(\Omega)}^2 
& = & \norm{e^{-t A} v_0}_{L^2(\Omega)}^2 + \norm{e^{-t A} A^{1 \slash 2} v_0}_{L^2(\Omega)}^2 \nonumber \\
& \leq & \norm{v_0}_{L^2(\Omega)}^2 + \norm{A^{1 \slash 2} v_0}_{L^2(\Omega)}^2 \nonumber \\
& \leq  & \norm{v_0}_{H^1(\Omega)}^2. \label{eq-esdu1}
\end{eqnarray}
Here, we used the fact that the operator $e^{-t A}$ is linear bounded in $L^2(\Omega)$ for all $t \in [0,T]$, with
$\norm{e^{-t A}}_{\mathcal B(L^2(\Omega))} \leq e^{-t c(\omega)} \leq 1$.

Similarly, since $f(s) \in D(A^{1 \slash 2})$ for a.e. $s \in (0,T)$, we get
\begin{equation}
\label{eq-cty2} 
t \mapsto \int_0^t e^{-(t -s)A} f(s)\ ds \in \mathcal C([0,T];H^1_0(\Omega)), 
\end{equation}
by arguing as before. Next, taking into account that
\begin{eqnarray*}
\norm{\int_0^t e^{-(t -s)A} f(s)\ ds}_{H^1(\Omega)}^2 
& = & \norm{\int_0^t e^{-(t -s)A} f(s)\ ds}_{L^2(\Omega)}^2 + \norm{A^{1 \slash 2} \int_0^t e^{-(t -s)A} f(s)\ ds}_{L^2(\Omega)}^2 \\
& = & \norm{\int_0^t e^{-(t -s)A} f(s)\ ds}_{L^2(\Omega)}^2 + \norm{\int_0^t e^{-(t -s)A} A^{1 \slash 2} f(s)\ ds}_{L^2(\Omega)}^2,
\end{eqnarray*}
for every $t \in [0,T]$, we obtain
\begin{eqnarray}
\norm{\int_0^t e^{-(t -s)A} f(s)\ ds}_{H^1(\Omega)}^2 & \leq & T \int_0^t \left( \norm{e^{-(t -s)A} f(s)}_{L^2(\Omega)}^2+ \norm{e^{-(t -s)A} A^{1 \slash 2} f(s)}_{L^2(\Omega)}^2 \right)\ ds \nonumber \\
& \leq  & T \int_0^t \left( \norm{f(s)}_{L^2(\Omega)}^2+ \norm{A^{1 \slash 2} f(s)}_{L^2(\Omega)}^2 \right)\ ds \nonumber\\
& \leq & T \norm{f}_{L^2(0,T;H^1(\Omega))}^2. \label{eq-esdu2}
\end{eqnarray}
Finally, putting \eqref{eq-du}, \eqref{eq-cty1} and \eqref{eq-cty2} together, we find that $v \in \mathcal C([0,T];H^1_0(\Omega))$, and \eqref{es1} follows readily from \eqref{eq-du}, \eqref{eq-esdu1} and \eqref{eq-esdu2}.


\subsection{Proof of Corollary \ref{cor-e}}
Since $F \in L^2(0,T;H_0^1(\Omega))$ and the initial data of the IBVP \eqref{eq:bvp} is zero everywhere, then there exists a unique solution $u \in H^{2,1}(Q)$ to \eqref{eq:bvp}, according to Theorem \ref{t1}. Moreover, differentiating \eqref{eq:bvp} with respect to $t$, we obtain that $\partial_t u$ is solution to the system \eqref{eq:bvp1}, associated with $f(t,x)=\sigma'(t) \beta(x) \in L^2(0,T;H_0^1(\Omega))$ and $v_0:=\sigma(0) \beta \in H_0^1(\Omega)$. Therefore, we have $\partial_t u \in H^{2,1}(Q) \cap \mathcal C([0,T];H^1_0(\Omega))$, from Theorem \ref{t1}, and the estimate \eqref{es1} yields
\begin{eqnarray*}
\norm{\partial_t u(t)}_{H^1(\Omega)} & \leq & \vert \sigma(0) \vert \norm{\beta}_{H^1(\Omega)} + \norm{\sigma'(t) \beta(x)}_{L^2(0,T;H^1(\Omega))} \nonumber \\
& \leq & \left(  \vert \sigma(0) \vert + T^{1 \slash 2} \norm{\sigma'}_{\mathcal C^0([0,T])} \right) \norm{\beta}_{H^1(\Omega)},
\end{eqnarray*}
for all $t \in [0,T]$. This leads to \eqref{es2}.


\section{A parabolic Carleman inequality in unbounded cylindrical domains}
\label{sec-CI}
In this section we establish a Carleman estimate, stated in Theorem \ref{t:CI}, for the heat operator
\begin{equation}
\label{eq:0}
P u := (\partial_t  - \Delta) u,\ u \in \mathcal{C}([0,T];H_0^1(\Omega)) \cap H^{2,1}(Q).
\end{equation}


\subsection{A Carleman estimate for $P$}

Let  $\gamma := \gamma' \times  \br$ be the same as in Theorem \ref{t2}. With reference to \cite[Lemma 1.1]{fur} (see also \cite[Lemma 1.1]{cha}), we pick\footnote{This is where the assumption that $\partial \Omega$ be ${\mathcal C}^4$ is needed.} a function $\psi_0 \in\mathcal{C}^4(\overline{\omega})$, such that
\begin{enumerate}[(c.i)]
\item $\psi_0(x')> 0$ for all $x' \in \overline{\omega}$;
\item $\exists \alpha_0>0$ such that $\vert \nabla' \psi_0(x') \vert \geq \alpha_0$ for all $x' \in \omega$;
\item $\partial_{\nu'} \psi_0(x') \le 0$ for all $x' \in \partial\omega \backslash \gamma'$.
\end{enumerate}
Here, $\nabla'$ denotes the gradient with respect to $x'=(x_1,\ldots,x_{n-1}) \in \br^{n-1}$, i.e. $\nabla' f:=(\partial_{x_1}f,\ldots,\partial_{x_{n-1}}f)$, and $\partial_{\nu'}$ is the normal derivative with respect to $\partial \omega$, that is $\partial_{\nu'}:=\nu' \cdot \nabla'$, where $\nu'$ stands for the outward normal vector to $\partial\omega$. 

Thus, putting $\psi(x)=\psi(x',x_n):=\psi_0(x')$ for all $x=(x',x_n)\in\overline{\Omega}$, it is apparent that the function $\psi\in\mathcal{C}^4(\overline{\Omega})\cap W^{4,\infty}(\Omega)$ satisfies the three following conditions:
\begin{enumerate}[(C.i)]
\item $\underset{x\in\Omega}{\inf}\psi(x) > 0$;
\item 
$\vert \nabla \psi(x) \vert  \geq \alpha_0 >0$ 
for all $x \in \overline{\Omega} $; 
\item $\dN \psi(x) \leq 0$ for all $x \in \Gamma \backslash \gamma$.
\end{enumerate}
Here and henceforth, the notation $\nu$ stands for the outward unit normal vector to the boundary 
$\Gamma$, and $\dN := \nu \cdot \nabla$. 
Evidently, $\nu = (\nu',0)$, so we have $\dN \psi = 
\partial_{\nu'} \psi_0$, as the function $\psi$ does not depend on the longitudinal 
variable $x_n$.

Next, for each $\rho \in (0,+\infty)$, we introduce the following weight function
\begin{equation}
\label{eq:1}
\pR (t,x) = \pR (t,x') := g(t) \left( e^{\rho \psi(x')} - 
e^{2 \rho \Vert \psi \Vert_{L^\infty(\Omega)}} \right)\ \text{with}\ g(t):= \frac{1}{t(T - t)},\ (t,x) \in Q, 
\end{equation}
and, for further use, we gather in the coming lemma several useful properties of $\pR$.

\begin{lem}
\label{lm-wf}
There exists a constant $\rho_0 \in (0,+\infty)$, depending only on $\psi$, such that for all $\rho \in [\rho_0,+\infty)$, the following statements hold uniformly in $Q$:
\begin{enumerate}[(a)]
\item 
$\vert \nabla \pR \vert \ge \alpha:= \frac{4\rho_0 \alpha_0}{T^2} >0$;
\item 
$\nabla \vert \nabla \pR \vert^2 \cdot \nabla \pR \geq 
C_0 \rho \vert \nabla \pR \vert^3$;
\item 
$\mathcal{H}(\pR) \xi \cdot \xi + C_1 \rho \vert \nabla \pR \vert \vert \xi 
\vert^2  \geq 0,\ \xi \in \br^{n - 1} \times \br$;
\item 
$\vert \partial_t \vert \nabla \pR \vert^2 \vert + \vert \Delta^2 \pR \vert + \vert \Delta \vert \nabla \pR \vert \vert + \rho^{-1} (\Delta \pR)^2\leq C_2 \vert \nabla \pR \vert^3$;
\item  
$\vert \partial_t^2 \pR \vert+  \vert \nabla \pR \vert^{-1} (\partial_t \pR)^2 \leq C_3 \lambda\ \vert \nabla \pR \vert^3,\ 
\lambda \geq  \lambda_0(\rho) := e^{4 \rho\|\psi\|_{L^\infty(\Omega)}}$.

\end{enumerate}
Here, $C_j$, $j=0,1,2,3$, are positive constants depending only on $T$, $\psi$ and $\alpha_0$, and $\mathcal{H}(\pR)$ denotes the Hessian
matrix of $\pR$ with respect to $x \in \Omega$.
\end{lem}

The proof of Lemma \ref{lm-wf} is postponed to Section \ref{sec-wf}.

Now, with reference to \eqref{eq:1}, we may state the Carleman estimate for the operator $P$, as follows.

\begin{theo}
\label{t:CI}
Let $u \in H^{2,1}(Q)\cap \mathcal C([0,T];H_0^1(\Omega))$ be real valued. 
Then, there exists $\rho_0 \in (0,+\infty)$, such that for all $\rho \in [\rho_0,+\infty)$, there is
$\lambda_0=\lambda_0(\rho) \in (0,+\infty)$, depending only 
$\alpha_0$, $\omega$, $\gamma'$, $T$ and $\rho$, such that the estimate
\begin{eqnarray}
& & \big \Vert e^{\lambda \pR}  (\lambda g)^{-1 \slash 2} \Delta u \big \Vert_{L^2(Q)} +
\big \Vert e^{\lambda \pR} (\lambda g)^{-1 \slash 2} \partial_t u \big \Vert_{L^2(Q)} 
+ \big \Vert e^{\lambda \pR} (\lambda g)^{1 \slash 2} \vert \nabla u \vert \big \Vert_{L^2(Q)} \nonumber \\
& & + \big \Vert e^{\lambda \pR} (\lambda g)^{\frac{3}{2}} u \big \Vert_{L^2(Q)} 
\le  C \left( \big \Vert e^{\lambda \pR} P u \big \Vert_{L^2(Q)} + 
\big \Vert e^{\lambda \pR} (\lambda g)^{1 \slash 2} \dN u \big \Vert_{L^2((0,T) \times \gamma)}
\right), \label{eq:CI}
\end{eqnarray}
holds for  all $\lambda \in [\lambda_0,+\infty)$ and some positive constant $C$, which depends only on $\alpha_0$, $\omega$, $\gamma'$, $T$, $\rho$ and $\lambda_0$. 
\end{theo}

We point out that it is actually possible to adapt the Carleman inequality of Theorem \ref{t:CI} to more general differential operators of the form 
$\partial_t-\Delta+A(t,x) \cdot \nabla +q(t,x)$, with $A\in L^\infty(Q)^n$ and $q\in L^\infty(Q)$. Nevertheless, as will appear in Section \ref{sec-prmain} below, such an estimate is not needed by the analysis of the inverse source problem carried out in this paper. Therefore, in order to avoid the inadequate expense of the size of this paper, we shall not go further into the matter.


\subsection{Proof of Theorem \ref{t:CI}}
\label{sec-prCI}

The proof of Theorem \ref{t:CI} is inspired by the ones of \cite[Lemma 1.2]{cha},  \cite[Lemma 5.2]{fc} and \cite[Theorem 3.4]{cho1}. For the sake of notational simplicity, we shall systematically omit the variables $t \in (0,T)$ and $x \in \Omega$, in front of the various functions appearing in this section.

Let $\rho$ and $\lambda$ be fixed in $(0,+\infty)$. We set $v := e^{\lambda \pR} u$, in such a way that 
$$
\int_Q e^{2 \lambda \pR} (P u)^2 dxdt = \int_Q (L_\lambda v)^2 dxdt\ \text{with}\ L_\lambda := e^{\lambda \pR} P e^{-\lambda \pR},
$$
and we split $L_\lambda$ into the sum $L_\lambda = L_\lambda^+ + L_\lambda^-$, where
\begin{equation}
\label{def-L}
L_\lambda^+ := - \Delta - \lambda \partial_t \pR - 
\lambda^2 \vert \nabla \pR \vert^2\ \text{and} \
L_\lambda^- := \partial_t + 2 \lambda \nabla \pR \cdot \nabla + \lambda \Delta 
\pR.
\end{equation}

\noindent a) The first step of the proof is to show that the estimate
\begin{eqnarray}
& & \int_Q (L_\lambda^+ v) L_\lambda^- v dxdt + \lambda \int_{(0,T) \times \gamma} 
\vert \dN \pR \vert (\dN v)^2 d\sigma dt \nonumber \\
& \geq & C_0 \lambda^3 \rho \int_{Q} \vert \nabla \pR \vert^3 v^2 dxdt - 2 C_1 \lambda \rho \int_{Q} 
\vert \nabla \pR \vert \vert \nabla v \vert^2 dxdt + R_0,
\label{eq:2,1}
\end{eqnarray}
holds uniformly in  $\rho \in [\rho_0,+\infty)$, with 
\begin{equation}
\label{eq:2,001}
R_0 := \lambda^2 \int_{Q} (\partial_t \vert \nabla \pR \vert^2) v^2 dxdt
+ \frac{\lambda}{2} \int_{Q} \left( \partial_t^2 \pR - \Delta^2 \pR \right) v^2 dxdt,
\end{equation}
where $\rho_0$ and the two constants $C_0$ and $C_1$ are the same as in Lemma \ref{lm-wf}. 
This can be done with the aid of \eqref{def-L}, involving
\begin{eqnarray}
(L_\lambda^+ v) L_\lambda^- v \
& = & -(\Delta v) \partial_t v - 2\lambda (\Delta v)
\nabla \pR \cdot \nabla v - \lambda (\Delta v) (\Delta \pR) v
- \lambda (\partial_t \pR) v \partial_t v \nonumber \\
& & - 2 \lambda^2 (\partial_t \pR)
(\nabla \pR \cdot \nabla v) v - \lambda^2 (\partial_t \pR) (\Delta \pR) v^2 - \lambda^2 \vert \nabla 
\pR \vert^2 v \partial_t v \nonumber \\
& & - 2\lambda^3 \vert \nabla \pR \vert^2 (\nabla \pR \cdot \nabla v) v - \lambda^3 \vert \nabla \pR \vert^2 (\Delta \pR) v^2, 
\label{eq:2,01}
\end{eqnarray}
through standard computations, and by integrating separately each term appearing in the right hand side of 
\eqref{eq:2,01}, with respect to $(t,x)$ over $Q$. 

The first term is easily treated with an integration by parts. Taking into account that $v$ vanishes on $\Sigma$ and in 
$\lbrace 0, T \rbrace \times \Omega$, we find that
\begin{equation}
\label{eq:2,2}
- \int_{Q} (\Delta v) \partial_t v \, dxdt = \frac{1}{2} \int_{Q} \partial_t
\vert \nabla v \vert^2 dxdt = 0.
\end{equation}
To examine the second term, we write
$$
\nabla (\nabla \pR \cdot \nabla v) = \mathcal{H}(\pR) \nabla v + 
\mathcal{H}(v) \nabla \pR,
$$
where, in accordance with the notation introduced in Lemma \ref{lm-wf}, $\mathcal{H}(v)$ stands for the Hessian
matrix of $v$, with respect to $x \in \Omega$, and integrate by parts, getting:
\begin{eqnarray}
& & - \int_{Q} (\Delta v) \nabla \pR \cdot \nabla v\ dxdt \nonumber \\
& = &
 \int_{Q} \nabla v \cdot \nabla (\nabla \pR \cdot \nabla v)\ dxdt 
-  \int_{\Sigma} (\dN v) \nabla \pR \cdot \nabla v\ d\sigma dt \nonumber \\
& = & \int_{Q} \mathcal{H}(\pR)\nabla v \cdot \nabla v\ dxdt 
+  \int_{Q} \mathcal{H}(v) \nabla \pR \cdot \nabla v\ dxdt
-  \int_{\Sigma} (\dN \pR) (\dN v)^2\ d\sigma dt. \label{eq:2,3}
\end{eqnarray}
In the last integral, we used the fact, arising from the identity $v_{\vert \Sigma} = 0$, that $\nabla v = (\dN v) \nu$ on $\Sigma$, or equivalently, that the tangential derivative of $v$ vanishes on $\Sigma$.
Moreover, as we have
\begin{eqnarray*}
\int_{Q} \mathcal{H}(v) \nabla \pR \cdot \nabla v\ dxdt & = & \frac{1}{2}
\int_{Q} \nabla \pR \cdot \nabla (\vert \nabla v \vert^2)\ dxdt \\
& = & - \frac{1}{2} \int_{Q} (\Delta \pR) \vert \nabla v \vert^2\ dxdt 
+ \frac{1}{2} \int_{\Sigma} (\dN \pR) \vert \nabla v \vert^2\ d\sigma dt \\
& = & - \frac{1}{2} \int_{Q} (\Delta \pR) \vert \nabla v \vert^2\ dxdt 
+ \frac{1}{2} \int_{\Sigma} (\dN \pR) ( \dN v)^2\ d\sigma dt,
\end{eqnarray*}
then we infer from \eqref{eq:2,3} that
\begin{eqnarray}
- 2 \lambda \int_{Q} (\Delta v) \nabla \pR \cdot \nabla v\ dxdt & = &
2 \lambda \int_{Q}  \mathcal{H}(\pR) \vert \nabla v \vert^2\ dxdt 
- \lambda \int_{Q} (\Delta \pR) \vert \nabla v \vert^2\ dxdt \nonumber \\
& & - \lambda \int_{\Sigma} (\dN \pR) (\dN v)^2\ d\sigma dt.\label{eq:2,4}
\end{eqnarray}
As for the third term entering the right hand side of  \eqref{eq:2,01}, we obtain
\begin{eqnarray}
- \lambda \int_{Q} (\Delta v) (\Delta \pR) v\ dxdt & = & \lambda \int_{Q} \nabla v 
\cdot \nabla \big( (\Delta \pR) v \big)\ dxdt \nonumber \\
& = & \frac{\lambda}{2} \int_{Q} \nabla (\Delta \pR) \cdot \nabla v^2\ dxdt
+ \lambda \int_{Q} (\Delta \pR) \vert \nabla v \vert^2\ dxdt \nonumber \\
& = & - \frac{\lambda}{2} \int_{Q} (\Delta^2 \pR) v^2\ dxdt + \lambda \int_{Q} 
(\Delta \pR) \vert \nabla v \vert^2\ dxdt,\label{eq:2,5}
\end{eqnarray}
by integrating by parts once more.
Next, we find that the fourth term reads
\begin{equation}
\label{eq:2,6}
- \lambda \int_{Q} (\partial_t \pR) v \partial_t v \, dxdt = 
- \frac{\lambda}{2} \int_{Q} (\partial_t \pR) \partial_t v^2 dxdt = 
\frac{\lambda}{2} \int_{Q} (\partial_t^2 \pR) v^2 dxdt,
\end{equation}
while the fifth one can be brought into the form 
\begin{equation}
\label{eq:2,7}
- 2\lambda^2 \int_{Q} (\partial_t \pR) v \nabla \pR \cdot \nabla v \, dxdt = 
- \lambda^2 \int_{Q} (\partial_t \pR) \nabla \pR \cdot \nabla v^2 dxdt.
\end{equation}
Quite similarly, the sixth and seventh terms may be rewritten as, respectively,
\begin{eqnarray}
- \lambda^2 \int_{Q} (\Delta \pR) (\partial_t \pR) v^2\ dxdt 
& = & \lambda^2 \int_{Q} \nabla \pR \cdot \nabla \big( (\partial_t \pR) v^2 \big)\ dxdt \label{eq:2,8} \\
& = & \lambda^2 \int_{Q} (\partial_t \pR) \nabla \pR \cdot \nabla v^2\ dxdt
+ \frac{\lambda^2}{2} \int_{Q} \big(  \partial_t \vert \nabla \pR \vert^2 \big) v^2\ dxdt, \nonumber
\end{eqnarray}
and
\begin{equation}
\label{eq:2,9}
- \lambda^2 \int_{Q} \vert \nabla \pR \vert^2 v \partial_t v\ dxdt = 
- \frac{\lambda^2}{2} \int_{Q} \vert \nabla \pR \vert^2 \partial_t v^2\ dxdt =
\frac{\lambda^2}{2} \int_{Q} \big( \partial_t \vert \nabla \pR \vert^2 \big)
v^2\ dxdt.
\end{equation}
Finally, the two last terms in the right hand side of \eqref{eq:2,01}, are re-expressed as
\begin{equation}
\label{eq:2,10}
- 2\lambda^3 \int_{Q} \vert \nabla \pR \vert^2 ( \nabla \pR \cdot 
\nabla v) v\ dxdt = - \lambda^3 \int_{Q} \vert \nabla \pR \vert^2 \nabla \pR 
\cdot \nabla v^2\ dxdt,
\end{equation}
and
\begin{eqnarray}
- \lambda^3 \int_{Q} (\Delta \pR) \vert \nabla \pR \vert^2 v^2\ dxdt & = & 
\lambda^3 \int_{Q} \nabla \pR \cdot \nabla \big( \vert \nabla \pR \vert^2 
v^2 \big)\ dxdt \label{eq:2,11} \\
& = & \lambda^3 \int_{Q}  \left( \nabla  \vert \nabla \pR \vert^2 \cdot
\nabla \pR \right) v^2\ dxdt
+ \lambda^3 \int_{Q} \vert \nabla \pR \vert^2 \nabla \pR \cdot \nabla v^2\ dxdt.
\nonumber
\end{eqnarray}
Thus, putting \eqref{eq:2,01}--\eqref{eq:2,11} together, we obtain that
\begin{eqnarray*}
&  & \int_{Q} (L_\lambda^+ v) L_\lambda^- v \ dxdt  + \lambda \int_{\Sigma} (\dN \pR) (\dN v)^2\ d\sigma dt \\
& = & \lambda^3 \int_{Q}  ( \nabla \vert \nabla \pR \vert^2  \cdot \nabla \pR \big) v^2\ dxdt + 
2 \lambda \int_{Q} \mathcal{H}(\pR) \vert \nabla v \vert^2\ dxdt  \\ 
& & + \frac{\lambda}{2} \int_{Q} \big( \partial_t^2 \pR - \Delta^2 \pR \big) v^2\ dxdt + \lambda^2 \int_{Q} \big( \partial_t  \vert \nabla \pR \vert^2 \big) v^2\ dxdt.
\end{eqnarray*}
With reference to Points (b) and (c) in Lemma \ref{lm-wf},  this entails for all $\rho \in [\rho_0,+\infty)$, that
\begin{eqnarray*}
& & \int_{Q} (L_\lambda^+ v) L_\lambda^- v\ dxdt + \lambda \int_{\Sigma} (\dN \pR) (\dN v)^2\ d\sigma dt \\
& \geq &  C_0 \lambda^3 \rho \int_{Q} \vert \nabla \pR \vert^3 v^2\ dxdt - 2 C_1 \lambda \rho \int_{Q} \vert 
\nabla \pR \vert^2 \vert \nabla v \vert^2\ dxdt + R_0, 
\end{eqnarray*}
where $R_0$ is given by \eqref{eq:2,001}. Finally, \eqref{eq:2,1} follows immediately from (C.iii) and the above estimate.

\noindent b) Further, remembering \eqref{def-L}, we see that
$\Delta v = -L_\lambda^+ v  - \lambda^2 \vert \nabla \pR \vert^2 v - \lambda (\partial_t \pR) v$, whence
\begin{equation}
\label{eq:2,14}
(\Delta v)^2  \leq 3 \left( (L_\lambda^+ v)^2 +  \lambda^4 \vert \nabla \pR \vert^4 v^2 
+  \lambda^2 (\partial_t \pR)^2 v^2 \right).
\end{equation}
Next, with reference to Point (a) in Lemma \ref{lm-wf}, we derive from \eqref{eq:2,14} that
\begin{equation}
\label{eq:2,141}
\big( \lambda \vert \nabla \pR \vert \big)^{-1}(\Delta v)^2 \leq 
\frac{3}{\alpha \lambda} (L_\lambda^+ v)^2 + 3 \lambda^3 \vert \nabla \pR \vert^3 v^2 
+ 3 \lambda \vert \nabla \pR \vert^{-1} (\partial_t \pR)^2 v^2,
\end{equation}
for all $\rho \in [\rho_0,+\infty)$. This entails that
\begin{equation}
\label{eq:2,15}
\int_Q \big( \lambda \vert \nabla \pR  \vert \big)^{-1} (\Delta v)^2\ dxdt \leq 
\frac{3}{\alpha \lambda} \int_Q (L_\lambda^+ v)^2\ dxdt + 
3 \int_Q \big( \lambda \vert \nabla \pR \vert \big)^3 v^2\ dxdt + R_1,
\end{equation}
with
\begin{equation}
\label{eq:2,150}
R_1 := 3 \lambda \int_Q \vert \nabla \pR \vert^{-1} (\partial_t \pR)^2 v^2\ dxdt.
\end{equation}
Further, taking into account that $v_{\vert \Sigma} = 0$, we get 
$$
 \int_Q \vert \nabla \pR \vert \vert \nabla v \vert^2\ dxdt = 
- \int_Q \vert \nabla \pR \vert v \Delta v\ dxdt + 
\frac{1}{2}  \int_Q  \big( \Delta \vert \nabla \pR \vert \big) v^2\ dxdt,
$$
upon integrating by parts. This and the estimate,
\begin{eqnarray*}
\rho^{3 \slash 2} \lambda \vert \nabla \pR \vert \vert v \Delta v \vert & = &
\left( \big( \lambda \vert \nabla \pR \vert \big)^{-1 \slash 2} \vert \Delta v \vert
\right) \left( \rho^{3 \slash 2}  \big( \lambda \vert \nabla \pR \vert \big)^{\frac{3}{2}} \vert v \vert
\right)  \\
& \leq & \frac{1}{2} \big( \lambda \vert \nabla \pR \vert \big)^{-1} (\Delta v)^2 +
\frac{\rho^3}{2} \big( \lambda \vert \nabla \pR \vert \big)^3 v^2,
\end{eqnarray*}
yield
\begin{eqnarray*}
\rho^{3 \slash 2} \lambda \int_Q \vert \nabla \pR \vert \vert \nabla v \vert^2\ dxdt & \leq &
\frac{1}{2} \int_Q \big(  \lambda \vert \nabla \pR \vert \big)^{-1} (\Delta v)^2\ dxdt +
\frac{\rho^3}{2} \int_Q \big( \lambda \vert \nabla \pR \vert \big)^3 v^2\ dxdt \\
& & + \frac{\rho^{3 \slash 2}}{2} \lambda \int_Q \big( \Delta  \vert \nabla \pR \vert \big) v^2\ dxdt.
\end{eqnarray*}
From this and \eqref{eq:2,15}-\eqref{eq:2,150}, it then follows that
\begin{eqnarray*}
& & \rho^{3 \slash 2} \lambda \int_Q \vert \nabla \pR \vert \vert \nabla v \vert^2\ dxdt +
\frac{1}{2} \int_Q \big( \lambda \vert \nabla \pR \vert \big)^{-1} (\Delta v)^2\ dxdt \\
& \leq & \frac{3}{\alpha \lambda} \int_Q (L_\lambda^+ v)^2\ dxdt + 
\left( 3 + \frac{\rho^3}{2} \right) \int_Q \big( \lambda \vert \nabla \pR \vert \big)^3 v^2\ dxdt  + R_2,
\end{eqnarray*}
with 
\begin{equation}
\label{eq:2,19}
R_2 := R_1 + \frac{\rho^{3 \slash 2}}{2} \lambda 
\int_Q \big( \Delta  \vert \nabla \pR \vert \big) v^2\ dxdt.
\end{equation}
Therefore, by substituting $\max(\rho_0,6^{1 \slash 3})$  for $\rho_0$, we obtain for all $\rho \in [\rho_0,+\infty)$, that
\begin{eqnarray*}
& & \rho^{1 \slash 2} \lambda \int_Q \vert \nabla \pR \vert  \vert \nabla v \vert^2\ dxdt +
\frac{1}{2} \int_Q \big( \lambda \vert \nabla \pR \vert \big)^{-1} (\Delta v)^2\ dxdt \\
& \leq & \frac{3}{\alpha \lambda} \int_Q (L_\lambda^+ v)^2\ dxdt + 
\rho^3 \int_Q \big( \lambda \vert \nabla \pR \vert \big)^3 v^2\ dxdt + R_2.
\end{eqnarray*}  
Putting this together with \eqref{eq:2,1}, we find 
\begin{eqnarray}
& & \frac{3}{\alpha \lambda} \int_Q (L_\lambda^+ v)^2\ dxdt + 
\frac{2}{C_0} \int_Q (L_\lambda^+ v) L_\lambda^- v\ dxdt  + \frac{2\lambda}{C_0} 
\int_{\gamma \times (0,T)} \vert \dN \pR \vert (\dN v)^2\ d\sigma dt \nonumber \\
& \geq & \left( \rho^{1 \slash 2} - \frac{4C_1}{C_0} \right) \rho \lambda 
\int_Q \vert \nabla \pR \vert \vert \nabla v \vert^2\ dxdt +
\rho \int_Q \big( \lambda \vert \nabla \pR \vert \big)^3 v^2\ dxdt \nonumber \\
& & + \frac{1}{2} \int_Q \big( \lambda \vert \nabla \pR \vert  \big)^{-1} (\Delta v)^2 \ dxdt + R_3,
\label{eq:2,21}
\end{eqnarray} 
with
\begin{equation}
\label{eq:2,22}
R_3 := \frac{2R_0}{C_0} - R_2.
\end{equation}
Then, upon possibly enlarging $\rho_0$, in such a way that $\rho_0^{1 \slash 2} - 4 (C_1  \slash C_0)$ is lower bounded by a positive constant $C_2$, we infer from \eqref{eq:2,21}, that for all $\rho \in [\rho_0,+\infty)$,\begin{eqnarray}
& & \frac{3}{\alpha \lambda} \int_Q (L_\lambda^+ v)^2\ dxdt + 
\frac{2}{C_0} \int_Q (L_\lambda^+ v) L_\lambda^- v\ dxdt  + \frac{2\lambda}{C_0} 
\int_{\gamma \times (0,T)} \vert \dN \pR \vert (\dN v)^2\ d\sigma dt \nonumber \\
& \geq & C_2 \lambda \int_Q \vert \nabla \pR \vert \vert \nabla v \vert^2\ dxdt +
\rho \int_Q \big( \lambda \vert \nabla \pR \vert \big)^3 v^2\ dxdt \nonumber \\
& & + \frac{1}{2} \int_Q \big( \lambda \vert \nabla \pR \vert \big)^{-1} (\Delta v)^2\ dxdt + R_3.
\label{eq:2,23}
\end{eqnarray} 

\noindent c) Further, since
$\partial_t v = L_\lambda^- v - 2 \lambda \nabla \pR \cdot \nabla v - \lambda (\Delta \pR) v$, 
by \eqref{def-L}, we find upon arguing as in the derivation \eqref{eq:2,141}, that
\begin{equation*}
\big( \lambda \vert \nabla \pR \vert \big)^{-1}(\partial_t v)^2 \leq 
\frac{3}{\alpha \lambda} (L_\lambda^- v)^2 + 12 \lambda \vert \nabla \pR \vert
\vert \nabla v \vert^2 + \frac{3\lambda}{\alpha} (\Delta \pR)^2 v^2,
\end{equation*}
which immediately entails
\begin{equation}
\label{eq:2,24}
\int_Q \big( \lambda \vert \nabla \pR \vert \big)^{-1}(\partial_t v)^2 dxdt \leq 
\frac{3}{\alpha \lambda} \int_Q (L_\lambda^- v)^2 dxdt 
+ \frac{12}{\alpha} \lambda \int_Q \vert \nabla \pR \vert \vert \nabla v \vert^2 dxdt + R_4,
\end{equation}
with
\begin{equation}
\label{eq:2,25}
R_4 := \frac{3\lambda}{\alpha} \int_Q (\Delta \pR)^2 v^2 dxdt.
\end{equation}
Putting \eqref{eq:2,23} together with \eqref{eq:2,24}, we see that
\begin{eqnarray*}
& & \frac{3}{\alpha \lambda} \int_Q (L_\lambda^+ v)^2 dxdt + 
\frac{2}{C_0} \int_Q (L_\lambda^+ v) L_\lambda^- v\ dxdt 
+ \frac{C_2}{8 \alpha \lambda} \int_Q (L_\lambda^- v)^2\ dxdt \\
& & + \frac{2\lambda}{C_0} \int_{\gamma \times (0,T)} \vert \dN \pR \vert 
(\dN v)^2\ d\sigma dt  \\
& \geq & \frac{C_2}{24} \int_Q ( \lambda \vert \nabla \pR \vert)^{-1} 
(\partial_t v)^2\  dxdt + 
\frac{1}{2} \int_Q ( \lambda \vert \nabla \pR \vert )^{-1} (\Delta v)^2\ dxdt 
+ \frac{C_2 \lambda}{2} \int_Q \vert \nabla \pR \vert \vert \nabla v \vert^2\  dxdt \\
&& + \rho  \int_Q ( \lambda \vert \nabla \pR \vert )^3 v^2\ dxdt + R_5,
\end{eqnarray*}
where we have set
\begin{equation}
\label{eq:2,27}
R_5 := R_3 - \frac{C_2}{24}R_4.
\end{equation}
Therefore, bearing in mind that $L_\lambda=L_\lambda^++L_\lambda^-$, we get that
\begin{eqnarray}
& & \frac{1}{C_0} \left( \int_Q (L_\lambda v)^2\ dxdt +  2 \lambda \int_{\gamma \times (0,T)} \vert \dN \pR \vert (\dN v)^2 d\sigma dt  \right) \nonumber \\
& \geq & \frac{C_2}{24} \int_Q ( \lambda \vert \nabla \pR \vert )^{-1} (\partial_t v)^2\ dxdt + 
\frac{1}{2} \int_Q ( \lambda \vert \nabla \pR \vert )^{-1} (\Delta v)^2\ dxdt  \nonumber\\
& & + \frac{C_2 \lambda}{2} \int_Q \vert \nabla \pR \vert \vert \nabla v \vert^2\ dxdt +
\rho  \int_Q  ( \lambda \vert \nabla \pR \vert )^3 v^2\ dxdt + R_5, \label{eq:2,28}
\end{eqnarray}
provided $\lambda \in [C_0 (3+C_2) \alpha^{-1},+\infty)$.

d) Without restricting the generality of the reasoning, we may assume in the sequel (upon possibly substituting $C_0 (3+C_2) \alpha^{-1}$ for $\lambda_0(\rho)$) that $\lambda_0(\rho) \in [C_0 (3+C_2) \alpha^{-1},+\infty)$. This way, putting
Points (d) and (e) in Lemma \ref{lm-wf}, together with \eqref{eq:2,001}, \eqref{eq:2,150}, \eqref{eq:2,19},
\eqref{eq:2,22}, and \eqref{eq:2,25} -\eqref{eq:2,27}, we find for any $\rho \in [\rho_0,+\infty)$ and $\lambda \in [\lambda_0(\rho),+\infty)$, that $\vert R_5 \vert$ is majorized by $\rho \lambda^2 \int_Q \vert \nabla \pR \vert^3 v^2\ dxdt$, up to some positive multiplicative constant, which is independent of $\rho$ and $\lambda$.
Therefore, upon possibly enlarging $\rho_0$, we obtain that
$$
\vert R_5 \vert \leq \frac{\rho}{2} \int_Q (  \lambda \vert \nabla \pR \vert )^3 v^2\ dxdt,\ \rho \in [\rho_0,+\infty),\ \lambda\in [\lambda_0(\rho),+\infty).
$$
This and \eqref{eq:2,28} yield
\begin{eqnarray}
& & \frac{1}{C_0} \left( \int_Q (L_\lambda v)^2\ dxdt +  2\lambda \int_{\gamma \times (0,T)} \vert \dN \pR \vert 
(\dN v)^2 d\sigma dt \right)  \nonumber \\
& \ge & \frac{C_2}{24} \int_Q ( \lambda \vert \nabla \pR \vert )^{-1} (\partial_t v)^2\ dxdt + 
\frac{1}{2} \int_Q ( \lambda \vert \nabla \pR \vert )^{-1} (\Delta v)^2\ dxdt  \nonumber\\
& & + \frac{C_2\lambda}{2} \int_Q \vert \nabla \pR \vert \vert \nabla v \vert^2\ dxdt +
 \frac{\rho}{2} \int_Q ( \lambda \vert \nabla \pR \vert )^3 v^2\ dxdt, \label{eq:2,29}
\end{eqnarray}
provided $\rho \in [\rho_0,+\infty)$ and $\lambda \in [\lambda_0(\rho),+\infty)$.
Next, since $v = e^{\lambda \pR} u$, then it holds true that
$e^{2\lambda \pR} \vert \nabla u \vert^2 \leq 2 \left( \vert \nabla v \vert^2 + 
\lambda^2 \vert \nabla \pR \vert^2 v^2 \right)$, and it ensues from \eqref{eq:2,29} and Point (a) in Lemma \ref{lm-wf}, that
\begin{eqnarray*}
& & \frac{1}{C_0} \left( \int_Q (L_\lambda v)^2\ dxdt + 2\lambda \int_{\gamma \times (0,T)} \vert \dN \pR \vert 
(\dN v)^2 d\sigma dt \right) \nonumber \\
& \ge & \frac{C_2}{24} \int_Q ( \lambda \vert \nabla \pR \vert )^{-1}  (\partial_t v)^2\ dxdt + 
\frac{1}{2} \int_Q ( \lambda \vert \nabla \pR \vert )^{-1} (\Delta v)^2\ dxdt  \\
& & + \frac{C_2\lambda}{4} \int_Q e^{2\lambda \pR}  \vert \nabla \pR \vert  
\vert \nabla u \vert^2\ dxdt +
\frac{\lambda^3}{2} (\rho - C_2 \alpha^{-1} ) \int_Q \vert \nabla \pR \vert^3 v^2\ dxdt. \nonumber
\end{eqnarray*}
Thus, upon possibly substituting $C_2 \alpha^{-1}+ 1$ for $\rho_0$, we immediately get for all $\rho \in [\rho_0,+\infty)$ and $\lambda \in [\lambda_0(\rho),+\infty)$, that
\begin{eqnarray}
& & \frac{1}{C_0} \left( \int_Q (L_\lambda v)^2\ dxdt + 2\lambda \int_{\gamma \times (0,T)} \vert \dN \pR \vert 
(\dN v)^2 d\sigma dt \right) \nonumber \\
& \ge & \frac{C_2}{24} \int_Q ( \lambda \vert \nabla \pR \vert )^{-1}  (\partial_t v)^2\ dxdt + 
\frac{1}{2} \int_Q ( \lambda \vert \nabla \pR \vert )^{-1} (\Delta v)^2\ dxdt  \nonumber \\
& & + \frac{C_2\lambda}{4} \int_Q e^{2\lambda \pR}  \vert \nabla \pR \vert  
\vert \nabla u \vert^2\ dxdt +
\frac{\lambda^3}{2} \int_Q \vert \nabla \pR \vert^3 v^2\ dxdt. \label{eq:2,31}
\end{eqnarray}
Further, due to \eqref{eq:1} and (C.ii), we have
$\vert \dN \pR \vert \leq C_4 g$ and $\vert \nabla \pR \vert^i \geq C_5 g^i$ for  $i \in \{  -1, 1, 3 \}$, in $Q$, where the two constants $C_4$ and $C_5$ depend only on $\rho$, $\psi$ and $\alpha_0$. Therefore, for all $\rho \in [\rho_0,+\infty)$ and $\lambda \in [\lambda_0(\rho),+\infty)$, we deduce from \eqref{eq:2,31} the existence of a positive constant
$C_6$, depending only on $\rho$, $T$, $\psi$ and $\alpha_0$, such that
\begin{eqnarray*}
& & \int_Q (L_\lambda v)^2\ dxdt + \int_{\gamma \times (0,T)} (\lambda g) (\dN v)^2 d\sigma dt \nonumber \\
& \geq &  C_6  \left( \int_Q (\lambda g)^{-1} (\partial_t v)^2\ dxdt + 
\int_Q (\lambda g)^{-1} (\Delta v)^2\ dxdt \right. 
\\
& &
\left. + \int_Q e^{2\lambda \pR} (\lambda g) \vert \nabla u \vert^2\ dxdt +
\int_Q (\lambda g)^3 v^2\ dxdt \right). \nonumber 
\end{eqnarray*}

Finally, \eqref{eq:CI} follows immediately from this upon remebering that
$v = e^{\lambda \pR} u$, $v_{\vert \Sigma} = 0$ and
$(L_\lambda v)^2 = e^{2 \lambda \pR} (P u)^2$.


\subsection{Proof of Lemma \ref{lm-wf}}
\label{sec-wf}
As in Section \ref{sec-prCI}, we systematically omit the variables $t \in (0,T)$ and $x \in \Omega$ in front of the functions $g$, $\psi$ or $\pR$, and their derivatives, in this proof.

\noindent a) We have 
\begin{equation}
\label{ab:1}
\nabla \pR = \rho g e^{\rho \psi}  \nabla \psi,
\end{equation}
from the very definition \eqref{eq:1} of $\pR$. Thus, using that $g(t) \geq 4T^{-2}$ for all $t \in (0,T)$, we deduce the estimate (a) directly from \eqref{ab:1} and (C.ii). 

\noindent b) With reference to  \eqref{eq:1} and \eqref{ab:1}, we see that
$$
\nabla \vert \nabla \pR \vert  = \rho g \nabla  \left( e^{\rho \psi} \vert \nabla \psi \vert \right) 
= 
\rho g \rho e^{\rho \psi} \left( \vert \nabla \psi \vert \nabla \psi + \nabla \vert \nabla \psi \vert \right) 
= \rho \vert \nabla \psi \vert   \nabla \pR  + \vert \nabla \pR \vert \frac{\nabla  \vert \nabla \psi \vert}{\vert \nabla \psi \vert},
$$
and hence
\begin{eqnarray}
\nabla  \vert \nabla \pR \vert^2 \cdot \nabla \pR 
& = & 2 \vert \nabla \pR \vert \nabla \vert \nabla \pR \vert \cdot \nabla \pR \nonumber \\
& = &  2 \left( \rho \vert \nabla \pR \vert^3 \vert \nabla \psi \vert + \vert \nabla \pR \vert^2  
\frac{\nabla  \vert \nabla \psi \vert \cdot \nabla \pR}{\vert \nabla \psi \vert} \right).
\label{ab:4}
\end{eqnarray}
Next, we have $\frac{\big\vert \nabla  \vert \nabla \psi \vert \cdot \nabla \pR \big\vert}{\vert \nabla \psi \vert} \leq C \alpha_0^{-1} \vert \nabla \pR \vert$, by (C.ii), where $C=C(\psi)$ is a positive constant depending only on $\psi$, in virtue of (C.ii). Therefore, \eqref{ab:4} yields
$$
\nabla  \vert \nabla \pR \vert^2  \cdot \nabla \pR 
\geq 2 \rho \alpha_0 \left( 1 - \frac{C}{\rho \alpha_0^2} \right) \vert \nabla \pR \vert^3,
$$
and the estimate (b) follows readily from this upon taking $\rho_0 \in [2 \alpha_0^2,+\infty)$.

\noindent c) For all natural numbers $i$ and $j$ in $\{ 1,2,\ldots, n\}$, we get that
$$
\partial_{x_i} \partial_{x_j} \pR  = \rho g  e^{\rho \psi} \left( \partial_{x_i} \partial_{x_j} \psi
+ \rho (\partial_{x_i} \psi) (\partial_{x_j} \psi) \right)  =  \left( \partial_{x_i} \partial_{x_j} \psi
+ \rho (\partial_{x_i} \psi) (\partial_{x_j} \psi) \right) \frac{\vert \nabla \pR \vert}{\vert \nabla \psi \vert},
$$
directly from \eqref{eq:1} and \eqref{ab:1}. Thus, using (C.ii), each
$\vert \partial_{x_i} \partial_{x_j} \pR \vert$, $i,j=1,\ldots,n$, is upper bounded by  $\rho \vert \nabla \pR \vert$, up to some positive multiplicative constant depending only on $\psi$ and $\alpha_0$. As a consequence, there exists $C=C(\psi,\alpha_0) \in (0,+\infty)$, such that
$$ \big\vert \mathcal{H}(\pR) \xi \cdot \xi \big\vert \leq C \rho \vert \nabla \pR \vert \vert \xi \vert^2,\ \xi \in \br^{n - 1} \times \br. $$
This immediately leads to (c).

\noindent d) In light of \eqref{ab:1}, we have
\begin{equation}
\label{wf2}
\rho g e^{\rho \psi} \abs{\nabla \psi} =\abs{\nabla \pR},
\end{equation}
so it follows from (C.i) and (C.ii), that
\begin{equation}
\label{wf2b}
g \leq \frac{\abs{\nabla \pR}}{\alpha_0},\ \rho \in [1,+\infty),
\end{equation}
and from the estimate $g \geq 4 \slash T^2$, that
\begin{equation}
\label{wf2c}
\rho e^{\rho \psi} \leq \frac{T^2}{4} \abs{\nabla \pR}.
\end{equation}
Therefore, as $\partial_t \abs{\nabla \pR}^2 = 2(2t-T) \rho^2 g^3 e^{2 \rho \psi} \abs{\nabla \psi}^2$, we infer from
\eqref{wf2}-\eqref{wf2b} that 
\begin{equation}
\label{wf3}
\abs{\partial_t \abs{\nabla \pR}^2} \leq C(T,\alpha_0) \abs{\nabla \pR}^3.
\end{equation}
Next, we find
\begin{equation}
\label{wf4}
\Delta \pR = \rho g e^{\rho \psi} \left( \Delta \psi + \rho \abs{\nabla \psi}^2 \right),
\end{equation}
by direct calculation, whence
\begin{equation}
\label{wf5}
\rho^{-1} (\Delta \pR)^2 \leq C(\psi,\alpha_0) \rho \abs{\nabla \pR}^2 \leq C(T,\psi,\alpha_0) \abs{\nabla \pR}^3,
\end{equation}
from \eqref{wf2} and \eqref{wf2c}. Quite similarly, we have
$$
\Delta \abs{\nabla \pR} = \rho g e^{\rho \psi} \left( \rho^2 \abs{\nabla \psi}^3 + \rho (\Delta \psi) \abs{\nabla \psi} + 2 \rho \nabla \psi \cdot \nabla \abs{\nabla \psi} + \Delta \abs{\nabla \psi} \right),
$$
by \eqref{wf2}, so we get
\begin{equation}
\label{wf6}
\abs{\Delta \abs{\nabla \pR} } \leq C(\psi) \rho^3 ge^{\rho \psi} \leq C(\psi,T) \abs{\nabla \pR}^3,
\end{equation}
with the help of \eqref{wf2} and \eqref{wf2c}. Last, with reference to \eqref{wf4}, we see that
$$
\Delta^2 \pR = \rho g e^{\rho \psi} \left( \rho  \nabla \psi + \nabla \right) \cdot \left( \rho \left( \Delta \psi + 
\rho \abs{\nabla \psi}^2 \right) \nabla \psi + \rho \nabla \abs{\nabla \psi}^2 + \nabla \Delta \psi \right),
$$
which, combined with (C.i), yields
$$
\abs{\Delta^2 \pR} \leq C(\psi,\alpha_0) \rho^4 g e^{\rho \psi} \leq C(\psi,T,\alpha_0) (\rho e^{-\rho \psi}) \rho^3 g e^{2 \rho \psi}  \leq C(\psi,T,\alpha_0)  \rho^3 g e^{2 \rho \psi}. $$
Thus, we have
$$ \abs{\Delta^2 \pR} \leq C(\psi,T,\alpha_0) \abs{\nabla \pR}^3, $$
and inequality (d) follows immediately from this, \eqref{wf3} and \eqref{wf5}-\eqref{wf6}.

\noindent e) Using \eqref{eq:1}, we get through direct computations, that
$$ \partial_t \pR=(2t-T)g \pR\ \mbox{and}\ \partial_t^2 \pR(t,x)=2 \left( 1 + (2t-T)^2 g \right) g \pR. $$
From this and the basic estimate $\abs{\pR} \leq g e^{2 \rho \norm{\psi}_{L^\infty(\Omega)}}$, we derive that
\begin{equation}
\label{wf1}
\abs{\partial_t \pR} \leq T g^2 e^{2 \rho \norm{\psi}_{L^\infty(\Omega)}}\ \mbox{and}\
\abs{\partial_t^2 \pR} \leq \frac{2(4+T^4)}{T^2} g^3 e^{2 \rho \norm{\psi}_{L^\infty(\Omega)}}.
\end{equation}
Further, as $\nabla \pR = \rho g e^{\rho \psi} \nabla \psi$, we have $\abs{\nabla \pR} \geq \rho g \alpha_0$, by (C.ii),
then it follows readily from \eqref{wf1}, that
$$
\abs{\partial_t^2 \pR} + \abs{\nabla \pR}^{-1} (\partial_t \pR)^2  \leq  C(T,\alpha_0) g^3 e^{4 \rho \norm{\psi}_{L^\infty(\Omega)}} \leq C(T,\alpha_0) \abs{\nabla \pR}^3 e^{4 \rho \norm{\psi}_{L^\infty(\Omega)}},\ \rho \in [1,+\infty), $$ 
giving the estimate (e).


\section{Proof of Theorem \ref{t2}}
\label{sec-prmain}

The proof of Theorem \ref{t2} is based on an observability inequality, that is derived in the coming section with the aid of the Carleman estimate \eqref{eq:CI}.


\subsection{Observability inequality}
\label{sec-obs}
The statement we are aiming for, is as follows.

\begin{prop}
\label{p1}
Let $\gamma'$ and $\gamma$ be the same as in Theorem \ref{t2}. For $v_0 \in H_0^1(\Omega)$,  let $v$ be the  $H^{2,1}(Q) \cap \mathcal{C}([0,T],H_0^1(\Omega))$-solution given by Theorem \ref{t1}, to the IBVP \eqref{eq:bvp1} associated with $f=0$. 
Then, there exists a constant $C > 0$, depending only on $\alpha_0$, $\omega$, $\gamma'$, and $T$, such that we have
\begin{equation}
\label{p1a}
\norm{v(T,\cdot)}_{H^1(\Omega)} \leq C \norm{\dN v}_{L^2((0,T) \times \gamma)}.
\end{equation}
\end{prop}
\begin{proof}
Let us pick a function $\eta \in \mathcal{C}^\infty([0,T];[0,1])$, obeying
$\eta(t) =0$ if $t \in \left[ 0,\frac{T}{4} \right]$, and $\eta(t)= 1$ if $t \in \left[ \frac{3T}{4},T \right]$, 
in such a way that
$w := \eta v$ is solution to the following system:
\begin{equation}
\label{o1}
\begin{cases} 
\partial_t w - \Delta w = \eta' v & \text{in} \quad Q, \\ 
w(0,\cdot) = 0 & \text{in} \quad \Omega,\\
w = 0 & \text{on} \quad \Sigma.
\end{cases}
\end{equation}
Then, we have 
$\Vert w(T,\cdot) \Vert_{H^1(\Omega)} \leq \norm{\eta' v}_{L^2(0,T;H^1(\Omega))}$, from  \eqref{es1},  and consequently
\begin{equation}
\label{eq:3,1}
\Vert v(T,\cdot) \Vert_{H^1(\Omega)}  = \Vert w(T,\cdot) \Vert_{H^1(\Omega)}\leq 
\Vert \eta \Vert_{W^{1,\infty}(0,T)} \Vert v \Vert_{H^1 \left( \left( \frac{T}{4},\frac{3T}{4} \right) \times \Omega \right)},
\end{equation}
since $\eta(T)=1$ and $\eta'$ vanishes in $\left( 0,\frac{T}{4} \right) \cup  \left( \frac{3T}{4}, T \right)$.

The next step of the proof is to apply the Carleman estimate \eqref{eq:CI} to $v$.  For $\rho=\rho_0$ and $\lambda =\lambda_0(\rho)$, we get that
\begin{equation}
\label{o2}
\Vert e^{\lambda \pR} (\lambda g)^{\frac{3}{2}} v \Vert_{L^2(Q)} + \Vert e^{\lambda \pR} (\lambda g)^{1 \slash 2} \vert \nabla v \vert \Vert_{L^2(Q)} \leq  c_1 \Vert e^{\lambda \pR} (\lambda g)^{1 \slash 2} \dN v \big \Vert_{L^2((0,T) \times \gamma)},
\end{equation}
where $c_1$ is a positive constant depending only on $\alpha_0$, $\omega$, $\gamma'$, and $T$.

Further, since $g(t) \ge \frac{4}{T^2}$ for all $t \in (0,T)$, and $\pR(t,x) \ge -\frac{16}{3 T^2} e^{2 \rho \Vert \psi \Vert_{L^\infty(\Omega)}}$ for every $(t,x) \in \left( \frac{T}{4},\frac{3T}{4} \right) \times \Omega$, by \eqref{eq:1}, we see that
\begin{eqnarray}
\label{o3}
& & \left( \frac{4 \lambda}{T^2} \right)^{1 \slash 2} e^{- \frac{16}{3 T^2}  \lambda e^{2 \rho \Vert \psi \Vert_{L^\infty(\Omega)}}} 
\left( \left( \frac{4 \lambda}{T^2} \right) \Vert v \Vert_{L^2 \left( \left( \frac{T}{4},\frac{3T}{4} \right) \times \Omega \right)} 
+ \Vert \vert \nabla v \vert \Vert_{L^2 \left( \left( \frac{T}{4},\frac{3T}{4} \right) \times \Omega \right)} \right) \nonumber \\
& \le & \Vert e^{\lambda \pR} (\lambda g)^{\frac{3}{2}} v \Vert_{L^2(Q)} + \Vert e^{\lambda \pR} (\lambda g)^{1 \slash 2} \vert \nabla v \vert \Vert_{L^2(Q)}.
\end{eqnarray}
On the other hand, for all $(t,x) \in (0,T) \times \gamma$, we have $\pR(t,x) \leq - c_2 g(t)$, where $c_2:= e^{2 \rho \Vert \psi \Vert_{L^\infty(\Omega)}} - e^{\rho \Vert \psi \Vert_{L^\infty(\Omega)}} \in (0,+\infty)$, according to \eqref{eq:1}, which entails
$$ 
\Vert e^{\lambda \pR} (\lambda g)^{1 \slash 2} \dN v \big \Vert_{L^2((0,T) \times \gamma)} \le c_3 \Vert \dN v \big \Vert_{L^2((0,T) \times \gamma)}, $$
with $c_3 := \sup_{r \in (0,+\infty)} r^{1 \slash 2} e^{-c_2 r } \in (0,+\infty)$. 

From this and \eqref{o2}-\eqref{o3}, it then follows that
\begin{equation}
\label{eq:3,2}
\Vert v \Vert_{H^1 \left( \Omega \times \left( \frac{T}{4},\frac{3T}{4} \right) \right)} 
\le c_4 \Vert \dN v \Vert_{L^2((0,T) \times \gamma)},
\end{equation}
where $c_4$ is a positive constant depending only on $\alpha_0$, $\omega$, $\gamma'$, and $T$.
Finally, we obtain Proposition \ref{p1} by combining \eqref{eq:3,1} with \eqref{eq:3,2}.
\end{proof}
Armed with Proposition \ref{p1}, we may now complete the proof of Theorem \ref{t2}.


\subsection{Completion of the proof}
\label{sec-cpr}

Due to \cite[Remark, pp. 367]{cho} and the hypothesis $\sigma(0) \neq 0$, we may assume without loss of generality, that $\sigma(t) = 1$ for all $t \in [0,T]$, in the sequel. 

\noindent a) {\it Time differentiation of the solution}. 
For $\beta \in H_0^1(\Omega)$, we denote by $u$ the $H^{2,1}(\Omega) \cap \mathcal{C}([0,T],H_0^1(\Omega))$-solution given by Theorem \ref{t1}, to the IBVP
\begin{equation*}
\begin{cases} 
\partial_t u - \Delta u = \beta & \text{in} \quad Q, \\ 
u(0,\cdot) = 0 & \text{in} \quad \Omega, \\
u = 0 & \text{on} \quad \Sigma,
\end{cases}
\end{equation*}
where the abuse of notation $\beta$ stands for the function $Q \ni (t,x) \mapsto \beta(x) \in L^2(0,T;H_0^1(\Omega))$. 
Then, $v := \partial_t u$ is a solution to the system
\begin{equation}
\label{eq:bvpp}
\begin{cases} 
\partial_t v - \Delta v = 0 & \text{in} \quad Q, \\ 
v(0,\cdot) = \beta & \text{in} \quad \Omega, \\
v = 0 & \text{on} \quad \Sigma,
\end{cases}
\end{equation}
and since $\beta \in H^1_0(\Omega)$, we find that $v\in H^{2,1}(Q) \cap \mathcal C([0,T];H^1_0(\Omega))$ by applying Theorem \ref{t1}. 

\noindent b) {\it Fiber and spectral decompositions}. We stick with the notations of Section \ref{sec-pr-t1} and still denote by $A$ the Dirichlet Laplacian in $L^2(\Omega)$, with domain $D(A)= H^1_0(\Omega)\cap H^2(\Omega)$.
Let $\mathcal F_{x_n}$ be the partial Fourier transform with respect to $x_n$, defined for all $\varphi \in L^1(Q)$, by
$$\mathcal F_{x_n} \varphi(t,x',k):=(2\pi)^{-{1 \slash 2}} \int_\br \varphi(t,x',x_n) e^{-ikx_n}dx_n,\ (t,x') \in(0,T) \times \omega,\ k \in \br, $$
and then suitably extended into a unitary operator in $L^2(Q)$. Due to the translational invariance of the operator $A$ in the infinite direction $x_n$, we have
\begin{equation}
\label{fiber}
\mathcal F_{x_n}  A \mathcal F_{x_n}^{-1} = \int_\br^{\oplus} H(k) dk,
\end{equation}
where $H(k):=A'+k^2$ and $A'$ is the Dirichlet Laplacian in $L^2(\omega)$. Since the multiplier by $k^2$ is bounded in $L^2(\omega)$, then we have $D(H(k))=D(A')=H_0^1(\omega) \cap H^2(\omega)$.

Further, as $\omega$ is bounded, we know that the resolvent of the operator $A'$ is compact, and that the spectrum of $A'$ is purely discrete. 
We denote by $\{ \lambda_\ell,\ \ell \geq1 \}$ the sequence of (positive) eigenvalues of $A'$, arranged in non-decreasing order and repeated according to their multiplicity. Next, we introduce a Hilbert basis $\{ \varphi_\ell,\ \ell \geq1 \}$ in $L^2(\omega)$, of eigenfunctions of $A'$, i.e.
$$ A' \varphi_\ell = \lambda_\ell \varphi_\ell,\ \ell \ge 1. $$
Then, putting $\hat{v}(\cdot,\cdot,k):=\mathcal F_{x_n}v(\cdot,\cdot,k)$ and $\hat{\beta}(\cdot,k):=\mathcal F_{x_n} \beta(\cdot,k)$ for all $k \in \br$, we find upon applying the transform
$\mathcal F_{x_n}$ to both sides of \eqref{eq:bvpp}, that
$$
\begin{cases} 
\partial_t \hat{v}(\cdot,\cdot,k) + H(k) \hat{v}(\cdot,\cdot,k) = 0 & \text{in} \quad  (0,T) \times \omega, \\ 
\hat{v}(0,\cdot,k) = \hat{\beta}(\cdot,k) &  \text{in}\quad \omega, \\
\hat{v}(\cdot,\cdot,k)= 0 & \text{on} \quad  (0,T) \times \partial\omega.
\end{cases}
$$
Therefore, each function $(0,T) \ni t \mapsto v_{k,\ell}(t):=\langle \hat{v}(t,\cdot,k), \varphi_\ell \rangle_{L^2(\omega)}$, for $(k,\ell) \in \br \times \bn^*:=\{ 1,2,\ldots \}$, is a solution to the differential equation
$v_{k,\ell}'+ (\lambda_\ell+k^2)v_{k,\ell} = 0$ in $(0,T)$, with initial data $v_{k,\ell}(0) = \beta_{k,\ell}:=\langle \hat{\beta}(\cdot,k) , \varphi_\ell \rangle_{L^2(\omega)}$.
As a consequence, we have
$$v_{k,\ell}(t)=\beta_{k,\ell} e^{-(\lambda_\ell+k^2)t},\ t \in [0,T],\ (k,\ell) \in \br \times \bn^*. $$
For further use, we notice from the above identity that
\begin{equation}
\label{co1}
\beta_{k,\ell}= e^{(\lambda_\ell+k^2) T} v_{k,\ell}(T),\ (k,\ell) \in \br \times \bn^*.
\end{equation}

\noindent c) {\it Splitting energies}.
Let us introduce the measure $\mu:=\sum_{\ell \in \mathbb N^*} \delta_\ell$ on $\br$, where $\delta_\ell$ denotes the delta Dirac measure at $\ell$, in such a way that $\|\beta\|_{L^2(\Omega)}^2 =\int_{\br^2} |\beta_{k,\ell}|^2d\mu(\ell)dk$, by the Plancherel formula. Thus, for all fixed $\lambda \in(\lambda_1,+\infty)$, it holds true that
$$
\|\beta\|_{L^2(\Omega)}^2 =\int_{\mathcal E_\lambda}|\beta_{k,\ell}|^2d\mu(\ell)dk+\int_{\br^2\setminus \mathcal E_\lambda}|\beta_{k,\ell}|^2d\mu(\ell)dk,
$$
where $\mathcal E_\lambda:=\{ (k,\ell) \in \br \times \mathbb N^*;\ \lambda_\ell + k^2 \le \lambda \}$ is the set of energies lower or equal to $\lambda$.
In light of \eqref{co1}, this entails that
\begin{eqnarray}
\|\beta\|_{L^2(\Omega)}^2 & \le & \int_{\mathcal E_\lambda} e^{2(\lambda_\ell+k^2) T} | v_{k,\ell}(T) |^2 d\mu(\ell)dk + \frac{1}{\lambda} \int_{\br^2}(\lambda_\ell+k^2)|\beta_{k,\ell}|^2d\mu(\ell)dk \nonumber \\
&\le & e^{2\lambda T} \int_{\br^2} | v_{k,\ell}(T) |^2 d\mu(\ell)dk+\frac{1}{\lambda} \int_{\br^2}(\lambda_\ell+k^2)|\beta_{k,\ell}|^2d\mu(\ell)dk \nonumber \\
&\le & e^{2\lambda T} \|v(T,\cdot)\|_{L^2(\Omega)}^2+\frac{1}{\lambda}\int_{\br^2}(\lambda_\ell+k^2)|\beta_{k,\ell}|^2d\mu(\ell)dk. \label{co2}
\end{eqnarray}
Further, we have
$\int_{\br^2}(\lambda_\ell+k^2) |\beta_{k,\ell}|^2d\mu(\ell) dk = \int_\br \| H(k)^{1 \slash 2} \hat{\beta}(\cdot,k) \|_{L^2(\omega)}^2  dk =  \| A^{1 \slash 2} \beta \|_{L^2(\Omega)}^2 = \| \nabla \beta \|_{L^2(\Omega)^n}^2$, in virtue of \eqref{fiber}, and consequently 
\begin{equation}
\label{eq:4,7}
\int_{\br^2}(\lambda_\ell+k^2) |\beta_{k,\ell}|^2d\mu(\ell) dk \le  \| \beta \|_{H^1(\Omega)}^2 \le M^2.
\end{equation} 
Moreover, by applying Proposition \ref{p1} to the solution $v$ of the IBVP \eqref{eq:bvpp}, we get that
$$
\Vert v(T,\cdot) \Vert_{H^1(\Omega)} \leq C  \Vert \dN v \Vert_{L^2((0,T) \times \gamma)} \leq C  \Vert \dN u \Vert_{H^1(0,T; L^2(\gamma))}.
$$
Putting this together with \eqref{co2}-\eqref{eq:4,7}, we obtain
\begin{equation}
\label{eq:4,8}
\Vert \beta \Vert_{L^2(\Omega)}^2 \le 
\left( C e^{2 \lambda T} \kappa^2 + \frac{M^2}{\lambda} \right),
\end{equation}
where $\kappa := \Vert \dN u \Vert_{H^1(0,T;L^2(\gamma))} \in [0,+\infty)$.

Having established \eqref{eq:4,8}, we investigate each of the three cases $\kappa=0$, $\kappa \in \left(0,e^{-2T\lambda_1}\right)$, and
$\kappa \in [e^{-2T\lambda_1},+\infty)$, separately. We start with $\kappa \in \left(0,e^{-2T\lambda_1}\right)$. In this case, taking $\lambda=-\frac{\ln \kappa}{2T}$ in \eqref{eq:4,8}, leads to
\begin{equation}
\label{eq:4,9}
\Vert \beta \Vert_{L^2(\Omega)}^2 \leq C\left(\kappa+|\ln(\kappa)|^{-1}\right)\leq C\left(\kappa^{{1 \slash 2}}+|\ln(\kappa)|^{-{1 \slash 2}}\right)^2, 
\end{equation}
which is exactly the statement \eqref{eq:in}. 
Moreover, we find $\beta=0$ by sending $\kappa$ to $0$ in \eqref{eq:4,9}, entailing \eqref{eq:in} with $\kappa=0$. Finally, for all $\kappa \in [e^{-2T\lambda_1},+\infty)$, it is apparent that
$$
\Vert \beta \Vert_{L^2(\Omega)}\leq M \leq Me^{T\lambda_1}\left(\kappa^{{1 \slash 2}}+|\ln(\kappa)|^{-{1 \slash 2}}\right). $$
Therefore, \eqref{eq:in} holds for any $\kappa \in [0,+\infty)$, and the proof of Theorem \ref{t2} is complete.


\end{document}